\documentclass{article}
\usepackage[margin=2cm]{geometry}
\usepackage[utf8]{inputenc} 
\usepackage[T1]{fontenc}    
\usepackage{lmodern}
\usepackage{hyperref}       
\usepackage{url}            
\usepackage{booktabs}       
\usepackage{amsfonts}       
\usepackage{xfrac}       
\usepackage{mathrsfs}         
\usepackage{amsmath,amssymb,bbm}
\usepackage{amsthm}
\usepackage{microtype}      
\usepackage{tikz-cd}

\newtheorem{theorem}{Theorem}[section]

\newtheorem{proposition}[theorem]{Proposition}

\newtheorem{corollary}[theorem]{Corollary}

\newtheorem{remark}{Remark}

\newcommand{\MW}{Milnor-Witt\ }
\newcommand{\rMW}{\mathrm{MW}}
\newcommand{\KM}{\mathrm{K}^\mathrm{M}}
\newcommand{\KMW}{\mathrm{K}^\mathrm{MW}}

\newcommand{\tbb}[1]{\widetilde{\mathbb{#1}}}
\newcommand{\wt}[1]{\widetilde{#1}}
\newcommand{\Spec}{\mathrm{Spec}\ }
\newcommand{\af}{\mathbb{A}}
\newcommand{\afnz}[1]{\mathbb{A}^{#1}\setminus \{0\}}

\newcommand{\tbZ}{\tbb{Z}}

\newcommand{\Gm}{\mathbb{G}_m}

\newcommand{\GL}{\mathrm{GL}}

\newcommand{\SH}{\mathcal{SH}}

\newcommand{\DMt}{\widetilde{\mathrm{DM}}}
\newcommand{\Mt}{\widetilde{\mathrm{M}}}

\newcommand{\HH}{\mathrm{H}_{\rMW}}
\newcommand{\HM}{\mathrm{H}_{\mathrm{M}}}
\newcommand{\etaHM}{ {^\eta\mathrm{H}_{\mathrm{M}}} }
\newcommand{\Heta}{\mathrm{H}_{\eta}}

\newcommand{\rHom}{\mathrm{Hom}}

\newcommand{\xr}[1]{\xrightarrow{#1}}

\title{\MW motivic decomposition of Stiefel varieties}
\author{Keyao Peng}
\date{}
\begin{document}
\maketitle

\section{Introduction}

This paper extends the research presented in \cite{peng2023milnorwitt}, which investigated the cohomology groups $\Heta^*(V_k(\af^n))$ with $\eta$-inverted coefficients by employing the $\af^1$-Leray spectral sequence \cite{asok2018homotopy}. For foundational results and notations relevant to this study, we refer the reader to \cite{peng2023milnorwitt}.

In this work, we will initially compute the integral MW-motivic cohomology groups $\HH^{p,q}(V_k(\af^n))$ associated with Stiefel varieties.

\begin{theorem}[{Theorem \ref{mainVk}}]
  Let $ i \in N_{n,k} $. The generators $ \alpha_{i} $ are elements of $ \HH^{2i+1,i+1}(V_k(\af^n)) $, and for $ J \in F_{n,k} $, the generators $ \beta_J $ belong to $ \HH^{d(J)}(V_k(\af^n)) $. For any subset $ I \subset N_{n,k} $, define 
  \[
  b_{I} = \prod_{J \in F_{n,k}, J \subset I} \beta_J \prod_{i \in I_T} \alpha_i \in \HH^{d(I)}(V_k(\af^n)).
  \]
  Then, we establish the following isomorphism of graded $ \HH^{*,*}(K) $-modules:
  \[
   \HH^{*,*}(V_k(\af^n)) \cong \bigoplus_{I \subset N_{n,k}} ( \HH )_I b_I,
  \]
  where
  \[
    ( \HH )_I= \begin{cases}
      \HH^{*,*}(K), & \text{if } I = I_F, \\
      \HM^{*,*}(K), & \text{if } \mathrm{Max}(I_T) \text{ is even}, \\
      \etaHM^{*,*}(K), & \text{if } \mathrm{Max}(I_T) \text{ is odd}. \\
    \end{cases}
  \]
\end{theorem}

It is important to note that our computations focus solely on the additive structures; the ring structures are considerably more complex.

Building upon these computations, we will proceed to establish the following MW-motive decomposition:

\begin{theorem}[{Theorem \ref{mwDecompVk}}]
Let $C_\eta(S)$ denote the cone of $\Mt(S) \xr{\eta}  \Mt(S)(1)[1]$ and define the following:

\begin{eqnarray*}
\wt{\mathbf{HS}}_{2k} & := & \Mt(\afnz{2k}) \cong \Mt(S) \oplus \Mt(S)(2k)[4k-1], \\
\wt{\mathbf{HS}}_{2k+1} & := & \Mt(S) \oplus C_\eta(S)(2k)[4k] \oplus \Mt(S)(4k+1)[8k].
\end{eqnarray*}

Then, for any $i, j \in \mathbb{N}$, we obtain an MW-motivic decomposition of the following forms:

\[
\Mt(V_{2j}(\af^{2i})) \cong \wt{\mathbf{HS}}_{2i} \otimes  \wt{\mathbf{HS}}_{2i-1} \otimes \wt{\mathbf{HS}}_{2i-3} \otimes \ldots \otimes \wt{\mathbf{HS}}_{2i-1-2(j-2)} \otimes \Mt(\afnz{2i+1-2j}),
\]

\[
\Mt(V_{2j+1}(\af^{2i})) \cong \wt{\mathbf{HS}}_{2i} \otimes \wt{\mathbf{HS}}_{2i-1} \otimes \wt{\mathbf{HS}}_{2i-3} \otimes \ldots \otimes \wt{\mathbf{HS}}_{2i-1-2(j-1)},
\]

\[
\Mt(V_{2j}(\af^{2i+1})) \cong \wt{\mathbf{HS}}_{2i+1} \otimes \wt{\mathbf{HS}}_{2i-1} \otimes \ldots \otimes \wt{\mathbf{HS}}_{2i+1-2(j-1)},
\]

\[
\Mt(V_{2j+1}(\af^{2i+1})) \cong \wt{\mathbf{HS}}_{2i+1} \otimes \wt{\mathbf{HS}}_{2i-1} \otimes \ldots \otimes \wt{\mathbf{HS}}_{2i+1-2(j-1)} \otimes \Mt(\afnz{2i+1-2j}).
\]
\end{theorem}

\subsection*{Notation}

Let $C_\eta(S)$ denote the cone of $\Mt(S) \xr{\eta}  \Mt(S)(1)[1]$. We sometimes use the index of degree $( p, \{q\} ) := (p+q ,q)$. A cohomology theory $\mathbb{E} \in \SH(K)$ is termed \textbf{coconnected} if $\mathbb{E}^{p, \{q\}}(K) = 0$ for $p > 0$.

\section{The $\af^1$-Thom-Gysin Sequence}

We will follow the computation from \cite{vcadek2003cohomology} to determine the total cohomology group $\HH^{*,*}(V_k(\af^n), \mathbb{Z})$ with integral coefficients as a module. The key element in this determination is the $\af^1$-Thom-Gysin sequence, which directly results from the $\af^1$-Leray spectral sequence.

\begin{theorem}[{\cite[Thm 5.2.8]{asok2018homotopy}}]
Let $f: X \to B$ be a Zariski locally trivial smooth fibration over a field $K$, where the fiber $F$ of $f$ is an $\af^1$-homology sphere (i.e., $\Sigma^r F \cong \Sigma^p \Gm^{\wedge q}$). Let $M$ be a homotopy module with a ring structure over $K$, and assume that $H^i_{\delta}(f_* M_X)$ is $K$-simple over $B$. Then there exists an Euler class $e(f) \in H^{p-r+1}(B, M_q)$, such that the Thom-Gysin sequence is given by:

\[
\cdots \to H^i(B, M_j) \to H^i(X, M_j) \to H^{i+r-p}(B, M_{j-q}) \xrightarrow{ \cup e(f)} H^{i+1}(B,M_j) \to \cdots
\]

where the connecting morphism is given by the product with $e(f)$.
\end{theorem}

In fact, we can generalize this to an Atiyah-Hirzebruch type spectral sequence by replacing $M$ with a spectrum $\mathbb{E} \in \SH(S)$, where $S$ is a Noetherian base scheme over the field $K$.

\begin{proposition}[{$\af^1$-Thom-Gysin sequence}]
  
  Let $\mathbb{E} \in \SH(S)$ be a ring spectrum, and let $f: V \to B$ be a rank-$n$ vector bundle that is $\mathbb{E}$-oriented. Let $X := V \setminus B$ denote the complement of the zero section, which constitutes a sphere bundle. Then, we acquire the Thom-Gysin exact sequence:
  \[
    \cdots \to \mathbb{E}^{i,j}(B) \to \mathbb{E}^{i,j}(X) \xrightarrow{r} \mathbb{E}^{i-(2n-1),j-n}(B) \xrightarrow{\cup e(f)} \mathbb{E}^{i+1,j}(B) \to \cdots
  \] 

  where $e(f) \in \mathbb{E}^{2n,n}(B)$ represents the Euler class of the vector bundle.
\end{proposition}

\begin{proof}
  
  Since $f: V \to B$ is oriented with respect to $\mathbb{E}$, we have the Thom isomorphism $\phi: \mathbb{E}^{i-2n,j-n}(B) \xrightarrow{\cong} \mathbb{E}^{i,j}(\mathrm{Th}(V))$, where $\mathrm{Th}(V) := V/X$ is the Thom space. Let $\tau = \phi(1) \in \mathbb{E}^{2n,n}(\mathrm{Th}(V))$ denote the Thom class. Furthermore, there exists a homotopy equivalence $f^*: \mathbb{E}^{i,j}(B) \xrightarrow{\cong} \mathbb{E}^{i,j}(V)$. Let the morphism $\mathrm{Th}(f)^*: \mathbb{E}^{i,j}(B) \xrightarrow{\cong} \mathbb{E}^{i,j}(\mathrm{Th}(V))$ be induced by the projection. We note that for the Thom isomorphism, $\phi(a) = \mathrm{Th}(f)^*(a) \cup \tau$. 

Now, consider the Gysin sequences associated with the cofiber sequence $X \to V \xrightarrow{q} \mathrm{Th}(V)$:
\[\begin{tikzcd}
    {\mathbb{E}^{i,j}(V)} & {\mathbb{E}^{i,j}(X)} & {\mathbb{E}^{i+1,j}(\mathrm{Th}(V))} & {\mathbb{E}^{i+1,j}(V)} \\
    {\mathbb{E}^{i,j}(B)} & {\mathbb{E}^{i,j}(X)} & {\mathbb{E}^{i-(2n-1),j-n}(B)} & {\mathbb{E}^{i+1,j}(B)}
    \arrow[from=1-1, to=1-2]
    \arrow[from=1-2, to=1-3]
    \arrow["{q^*}", from=1-3, to=1-4]
    \arrow["{f^*}"', from=2-1, to=1-1]
    \arrow[from=2-1, to=2-2]
    \arrow[Rightarrow, no head, from=2-2, to=1-2]
    \arrow[from=2-2, to=2-3]
    \arrow["\phi"', from=2-3, to=1-3]
    \arrow["\partial", from=2-3, to=2-4]
    \arrow["{f^*}"', from=2-4, to=1-4]
\end{tikzcd}\]

We only need to determine the morphism $\partial := (f^*)^{-1} q^* \phi$. By the properties of the Thom isomorphism, we have 

\[
\partial(a) = (f^*)^{-1} q^*(\mathrm{Th}(f)^*(a) \cup \tau) = (f^*)^{-1}(f^*(a) \cup q^*(\tau)) = a \cup (f^*)^{-1} q^*(\tau).
\]
By the definition of the Euler class, $ e(f) := (f^*)^{-1} q^*(\tau) $, we can conclude our proof.

\end{proof}

Pursuing the diagram further, we can derive the following corollary:

\begin{corollary}
  \label{fibEmb}
  We maintain the previously established notation. Let $ g: \afnz{n} \to X $ be a fiber embedding induced by the base point $ b: pt \to B $. Then the pullback $ g^*: \mathbb{E}^{i,j}(X) \to \mathbb{E}^{i,j}(\afnz{n}) $ is defined as follows:
  \[
    \mathbb{E}^{i,j}(X) \xrightarrow{r} \mathbb{E}^{i-(2n-1),j-n}(B) \xrightarrow{b^*} \mathbb{E}^{i-(2n-1),j-n}(pt) \xrightarrow{\cup \theta} \mathbb{E}^{i,j}(\afnz{n}),
  \]
  where $ \theta \in \mathbb{E}^{2n-1,n}(\afnz{n}) $ serves as a generator.
\end{corollary}

\section{MW-Motivic Cohomology Groups of $ V_k(\af^n) $}

We apply this theorem to the fiber sequences of Stiefel varieties over the field $ K $:

\[
\afnz{n-k} \to V_{k+1}(\af^n) \xrightarrow{f_{n,k}} V_{k}(\af^{n}).
\]

The fiber $ \afnz{n-k} \cong \Sigma^{n-k-1} \Gm^{ \wedge n-k} $ is precisely an $ \af^1 $-sphere and corresponds to a $ \HH $-oriented bundle, which is, in fact, $ \GL $-oriented as shown in \cite[\S 3.6, \S 3.7]{peng2023milnorwitt}. The Thom-Gysin sequence for MW motivic cohomology $ \HH^{i,j}(-) := \HH^{i,j}(-,\mathbb{Z}) $ is expressed as follows:

\[
\cdots \xrightarrow{\cup e(f_{n,k})} \HH^{i,j}( V_{k}(\af^n)) \to \HH^{i,j}( V_{k+1}(\af^n)) \to \HH^{i-(2(n-k)-1),j-(n-k)}( V_{k}(\af^n)) \xrightarrow{\cup e(f_{n,k})} \HH^{i+1,j}( V_{k}(\af^n)) \to \cdots.
\]

Consequently, we derive the short exact sequence:

\begin{equation}
  \label{exactVk}
  0 \to \mathrm{Coker}(\cup e(f_{n,k})) \to \HH^{i,j} \to \mathrm{Ker}(\cup e(f_{n,k})) \to 0.
\end{equation}

To carry out calculations, it suffices to understand the Euler class $ e(f_{n,k}) \in \HH^{2(n-k),n-k}(V_{k}(\af^n)) $.

We begin by reviewing some facts about the cone $ C_{\eta}(K) \in \DMt(K) $, which is associated with the morphism $ \Mt(K)(1)[1] \xrightarrow{\eta} \Mt(K) \to C_{\eta}(K) $.

\begin{proposition}[{\cite[Prop 2.2]{yang2022split}}]
   The following exact sequence splits:
   \[
   \HH^{i-2,j-1}(K) \to \HH^{i,j}(C_{\eta}(K)) \to \HH^{i,j}(K) \xrightarrow{\eta} 
   \]
   \[
   \HH^{i,j}(C_{\eta}(K)) = \begin{cases}
      \HM^{i-2,j-1}(K) \oplus \HM^{i,j}(K) , & \text{if } i \neq j \text{ or } i=j<0 ,\\
      \HM^{j-2,j-1}(K) \oplus 2 \KM_j(K) , & \text{if } i=j \geq 0\\
   \end{cases}
   \]
   Here, $ \HM^{i,j}(-) := \HM^{i,j}(-,\mathbb{Z}) $ denotes the motivic cohomology as defined by Voevodsky \cite{mazza2006lecture}.
\end{proposition}

We also note the following exact sequence of abelian groups:
\[ 
  0 \to 2\KM_j(K) \cong {^\eta \KMW_j(K)} \to \KMW_j(K) \xrightarrow{\eta} \KMW_{j-1}(K) \to \KM_{j-1}(K) \to 0.
\]
This motivates us to define a modified cohomology group:
\[
  \etaHM^{i,j}(K) = \begin{cases}
    \HM^{i,j}(K), & \text{if } i \neq j,\\
    2\KM_j(K), & \text{if } i=j.
  \end{cases}
\]  

Consequently, we have the following exact sequence of the graded cohomology group $ \HH^{*,*}(K) = \bigoplus_{i,j} \HH^{*,*}(K) $:
\[
  0 \to \etaHM^{*,*}(K) \to \HH^{*,*}(K) \xrightarrow{\eta} \HH^{*-1,*-1}(K) \to \HM^{*-1,*-1}(K) \to 0.
\]

Additionally, we have $ \HH^{*,*}(C_{\eta}(K)) = \HM^{*,*}(K) \theta_{2,1} \oplus \etaHM^{*,*}(K) $.

To compute the cohomology of $ V_2(\af^{n}) $, we utilize the following theorem:

\begin{proposition}[{\cite[Cor 3.9]{peng2023milnorwitt}}]
  \label{mwDecompV2} 
  The following motivic decompositions are valid in $ \DMt(S) $:
  \[ \Mt(V_2(\af^{n})) \cong \Mt(\afnz{n}) \otimes \Mt(\afnz{n-1}) \] for $ n $ even, and
  \[ \Mt(V_2(\af^{n})) \cong \wt{\mathbf{HS}}_{n} \] for $ n $ odd.
\end{proposition}

\begin{corollary}
  We have an isomorphism of graded $ \HH^{*,*}(K) $-modules:
\[
\HH^{*,*}(V_2(\af^n)) = \begin{cases}
    \HH^{*,*}(K)[\beta_{n-1}, \beta_{n-2}] / (\beta_{n-1}^2, \beta_{n-2}^2), & \text{if } n \text{ is even}, \\
    \HH^{*,*}(K) \beta_{n-1,n-2} \oplus \HM^{*,*}(K) \alpha_{n-1} \oplus \etaHM^{*,*}(K) \alpha_{n-2} \oplus \HH^{*,*}(K), & \text{if } n \text{ is odd}, \\
\end{cases}
\]
where $\mathrm{deg}(\beta_{i}) = \mathrm{deg}(\alpha_{i}) = (2i+1, i+1) = (i, \{i+1\})$, meaning that $\alpha_{i} \in \HH^{2i+1,i+1}(V_2(\af^n))$, and $\mathrm{deg}(\beta_{n-1,n-2}) = (4n-4, 2n-1) = (2n-3, \{2n-1\})$.

\end{corollary}

Now, we can state our main result for $ V_k(\af^n) $. Let $ N_{n,k} = \{ t \in \mathbb{Z} \mid n-k \leq t \leq n-1 \} $. We define the \textit{free pairs} as follows:
\[
F_{n,k} = \{ \{n-1\} \mid \text{if } n \text{ is even} \} \sqcup \{ \{i, i-1\} \subset N_{n,k} \mid \text{if } i \text{ is even} \} \sqcup \{ \{n-k\} \mid \text{if } n-k \text{ is even} \}.
\]
Additionally, we define the \textit{free part} and the \textit{torsion part} of a subset $ I \subset N_{n,k} $ as follows:
\[
I_F = \bigsqcup_{J \in F_{n,k}, J \subset I} J, \quad I_T = I \setminus I_F.
\]
Furthermore, we define the degree for a subset $ I $ as $ d(I) = \left(\sum_{i \in I} (2i + 1), \sum_{i \in I} (i + 1)\right) $.

\begin{theorem}
  \label{mainVk}
  Let $ i \in N_{n,k} $. The generators $ \alpha_{i} $ are elements of $ \HH^{2i+1,i+1}(V_k(\af^n)) $, and for $ J \in F_{n,k} $, the generators $ \beta_J $ belong to $ \HH^{d(J)}(V_k(\af^n)) $. For any subset $ I \subset N_{n,k} $, define 
  \[
  b_{I} = \prod_{J \in F_{n,k}, J \subset I} \beta_J \prod_{i \in I_T} \alpha_i \in \HH^{d(I)}(V_k(\af^n)).
  \]
  Then, we establish the following isomorphism of graded $ \HH^{*,*}(K) $-modules:
  \[
   \HH^{*,*}(V_k(\af^n)) \cong \bigoplus_{I \subset N_{n,k}} ( \HH )_I b_I,
  \]
  where
  \[
    ( \HH )_I= \begin{cases}
      \HH^{*,*}(K), & \text{if } I = I_F, \\
      \HM^{*,*}(K), & \text{if } \mathrm{Max}(I_T) \text{ is even}, \\
      \etaHM^{*,*}(K), & \text{if } \mathrm{Max}(I_T) \text{ is odd}. \\
    \end{cases}
  \]
\end{theorem}

We can compute the Euler class as a corollary:

\begin{corollary}
  \label{eulerVk}
  For $ k \geq 2 $, we have
  \[
  \HH^{2(n-k),n-k}(V_k(\af^n)) = \begin{cases}
    \KMW_{-1}\beta_{n-k}, & \text{if } n-k \text{ is even,}\\
    0, & \text{if } n-k \text{ is odd.}\\
  \end{cases}
  \]
  Consequently, the Euler class satisfies $ e(f_{n,k})= \eta \beta_{n-k} $ if $ n-k $ is even, and $ e(f_{n,k})= 0 $ if $ n-k $ is odd.
\end{corollary}

\begin{proof}
  
  The first part of the statement follows from comparing the degrees of generators in Theorem \ref{mainVk}. Specifically, since $ b_{n-k} $ has the lowest degree $ (n-k, \{n-k+1\}) $ and $ \HH $ is coconnected (i.e., $ \HH^{p, \{q\}}(K)=0 $ for $ p>0 $), we conclude that
  \[
    \HH^{2(n-k),n-k}(V_k(\af^n))= (\HH)_{n-k}^{-1,-1}b_{n-k}.
  \]
  For $ n-k $ even, we have $ b_{n-k}=\beta_{n-k} $ and $ (\HH)_{n-k}^{-1,-1} = \KMW_{-1} $. Conversely, for $ n-k $ odd, it follows that $ (\HH)_{n-k}^{-1,-1}=\etaHM^{-1,-1}(K)=0 $. 

  The known values for the Euler classes in the case $ k=1 $ are provided by \cite[Prop 3.8]{peng2023milnorwitt}. Specifically, in these situations, it is stated that $ e(f_{n,1}) = n_{\epsilon}\eta \beta_{n-1} \in \HH^{2(n-1),n-1}(\afnz{n}) $. This implies that $ e(f_{n,1})=0 $ if $ n $ is even, and $ e(f_{n,1})=\eta \beta_{n-1} $ if $ n $ is odd, where $ \beta_{n-1} \in \HH^{2n-1,n}(\afnz{n}) $ is a generator.  

We can consider the following diagram: 
\[
\begin{tikzcd}
    {\afnz{n-k}} & {V_{2}(\af^{n-k+1})} & {\afnz{n-k+1}} \\
    {\afnz{n-k}} & {V_{k+1}(\af^n)} & {V_{k}(\af^n)} \\
    & {V_{k-1}(\af^n)} & {V_{k-1}(\af^n)}
    \arrow[from=1-1, to=1-2]
    \arrow[Rightarrow, no head, from=1-1, to=2-1]
    \arrow["{f_{n-k+1,1}}", from=1-2, to=1-3]
    \arrow[from=1-2, to=2-2]
    \arrow["\lrcorner"{anchor=center, pos=0.125}, draw=none, from=1-2, to=2-3]
    \arrow["{g}", from=1-3, to=2-3]
    \arrow[from=2-1, to=2-2]
    \arrow["{f_{n,k}}", from=2-2, to=2-3]
    \arrow[from=2-2, to=3-2]
    \arrow[from=2-3, to=3-3]
    \arrow[Rightarrow, no head, from=3-2, to=3-3]
\end{tikzcd}
\]

Utilizing Corollary \ref{fibEmb}, we observe that when $ n-k $ is even, the pullback 

\[
g^*: \HH^{2(n-k),n-k}(V_{k}(\af^{n})) \to \HH^{2(n-k),n-k}(\afnz{n-k+1}) 
\]

is an isomorphism. Furthermore, the Euler class is compatible with the pullback, expressed as $ g^*(e(f_{n,k})) = e(f_{n-k+1,1}) $. Consequently, we derive that $ e(f_{n,k}) = \eta \beta_{n-k} $. 

\end{proof}

Although we do not provide the multiplicative structure of the graded cohomology group, we can state a useful fact about it.

\begin{corollary}
\label{etaSqZero}
Assume $ n-k $ is even, and let $ \beta_{n-k} \in \HH^{*,*}(V_k(\af^n)) $ be a generator with $ n > k $. Then we have 

\[
\eta \beta_{n-k}^2 = 0.
\]
\end{corollary}

\begin{proof}

By induction and by comparing the degree, we find that the degree of $ \beta_{n-k}^2 $ is given by $ \mathrm{deg}(\beta_{n-k}^2) = (2(n-k), \{2(n-k+1)\}) $. Again, exploiting the coconnectedness of $ \HH $, we conclude that $ \beta_{n-k}^2 \in (\HH)_{2(n-k)}^{1,1} b_{2(n-k)} = \KM_1 \alpha_{2(n-k)} $ if $ 2(n-k) < n $, or that $ \beta_{n-k}^2 = 0 $ if $ 2(n-k) \geq n $. 

In both scenarios, we establish that $ \eta \beta_{n-k}^2 = 0 $. 
\end{proof}

\begin{remark}

According to the computations in motivic cohomology \cite{williams2012motivic}, we obtain the following result:
\[
    \beta_{n-k}^2= \begin{cases}
      0, &\text{if } 2(n-k) \geq n,\\
      [-1]_{\KM}\alpha_{2(n-k)}, &\text{if } 2(n-k) < n.\\
    \end{cases}  
\]
\end{remark}

We prove our main theorem by induction.

\begin{proposition}
Assuming that Theorem \ref{mainVk} holds for $ V_k(\af^n) $, then the exact sequence (\ref{exactVk}) splits. Furthermore, we have the following isomorphism of $ \HH^{*,*}(K) $-modules:
\[
    \HH^{*,*}(V_{k+1}(\af^n)) \cong \begin{cases}
      \HH^{*,*}(V_k(\af^n)) \oplus \HH^{*,*}(V_k(\af^n)) \beta_{n-k-1}, &\text{if } n-k \text{ is odd,} \\
      \HH^{*,*}(V_{k-1}(\af^n)) \oplus \mathrm{Coker}_{\eta}(\HH^{*,*}(V_{k-1}(\af^n))) \alpha_{n-k} \\ 
      \oplus \mathrm{Ker}_{\eta}(\HH^{*,*}(V_{k-1}(\af^n))) \alpha_{n-k-1} \oplus \HH^{*,*}(V_{k-1}(\af^n)) \beta_{n-k,n-k-1}, &\text{if } n-k \text{ is even.}\\
    \end{cases}
\]
Here, $ \mathrm{Coker}_{\eta}(\HH^{*,*}(V_{k-1}(\af^n))) $ and $ \mathrm{Ker}_{\eta}(\HH^{*,*}(V_{k-1}(\af^n))) $ denote the cokernel and kernel of $ \eta $, respectively:
\[
    0\to \mathrm{Ker}_{\eta}(\HH^{*,*}(V_{k-1}(\af^n))) \to \HH^{*,*}(V_{k-1}(\af^n)) \xrightarrow{\eta} \HH^{*,*}(V_{k-1}(\af^n)) \to \mathrm{Coker}_{\eta}(\HH^{*,*}(V_{k-1}(\af^n))) \to 0.
\]
\end{proposition}

\begin{proof}

Since the boundary morphism is evident from Corollaries \ref{eulerVk} and \ref{etaSqZero}, it suffices to demonstrate that the exact sequence (\ref{exactVk}) splits.

When $ n-k $ is odd, the result is straightforward. In this case, $ \HH^{*,*}(V_{k+1}(\af^n)) $ serves as a $ \HH^{*,*}(V_k(\af^n)) $-module, and $ \HH^{*,*}(V_k(\af^n)) $ itself is a free $ \HH^{*,*}(V_k(\af^n)) $-module. We denote the generators by $ 1 $ and $ \beta_{n-k-1} $.

When $ n-k $ is even, under our assumption, we have the isomorphism 

\[
\HH^{*,*}(V_k(\af^n)) \cong \HH^{*,*}(V_{k-1}(\af^n)) \oplus \HH^{*,*}(V_{k-1}(\af^n)) \beta_{n-k}.
\]

From Corollary \ref{eulerVk}, we can identify the Euler class as $ e(f_{n,k}) = \eta \beta_{n-k} $. Moreover, invoking Corollary \ref{etaSqZero}, we can rewrite the exact sequence (\ref{exactVk}) as:

\begin{eqnarray*}
0 & \to & \HH^{*,*}(V_{k-1}(\af^n)) \oplus \mathrm{Coker}_{\eta}(\HH^{*,*}(V_{k-1}(\af^n))) \alpha_{n-k} \\
& \to & \HH^{*,*}(V_{k+1}(\af^n)) \\
& \to & \mathrm{Ker}_{\eta}(\HH^{*,*}(V_{k-1}(\af^n))) \alpha_{n-k-1} \oplus \HH^{*,*}(V_{k-1}(\af^n)) \beta_{n-k,n-k-1} \to 0.
\end{eqnarray*}

To demonstrate that this sequence splits, it is sufficient to consider the $\eta$-torsion part. This is evident when comparing it with the exact sequence of $\HM^{*,*}(V_{k+1}(\af^n))$ presented in \cite{williams2012motivic}.
\end{proof}

\begin{proof}[Proof of Theorem \ref{mainVk}]

We observe that the following relations hold:

\[
\mathrm{Coker}_{\eta}(\HH^{*,*}(K)) = \HM^{*,*}(K), \quad \mathrm{Ker}_{\eta}(\HH^{*,*}(K)) = \etaHM^{*,*}(K),
\]
\[
\mathrm{Coker}_{\eta}(\HM^{*,*}(K)) = \mathrm{Ker}_{\eta}(\HM^{*,*}(K)) = \HM^{*,*}(K),
\]
\[
\mathrm{Coker}_{\eta}(\etaHM^{*,*}(K)) = \mathrm{Ker}_{\eta}(\etaHM^{*,*}(K)) = \etaHM^{*,*}(K).
\]

These observations provide the necessary conditions to establish the existence of a direct summand in Theorem \ref{mainVk}.
\end{proof}

\section{Decomposition of MW-motive of $ V_k(\af^n) $}

It is noteworthy that throughout the proof, we relied solely on the coconnectedness of $\HH$ as described in Corollaries \ref{eulerVk} and \ref{etaSqZero}. Consequently, the same computation applies to all coconnected MW-motives $ \mathbb{E} \in \DMt(K) $. In particular, this assertion holds for $ C_{\eta} = C_{\eta}(K) $.

\begin{theorem}
  \label{EmainVk}
Let $\mathbb{E} \in \DMt(K)$ be a coconnected MW-motive, and let ${^\eta\mathbb{E}}$ and $\mathbb{E}/\eta$ denote the kernel and cokernel of the morphism $\eta$. The sequence is defined as follows:
\[
  0 \to {^\eta\mathbb{E}}^{*,*}(K) \to \mathbb{E}^{*,*}(K) \xr{\eta} \mathbb{E}^{*,*}(K)  \to (\mathbb{E}/\eta)^{*,*}(K) \to 0
\]
We use the same notation as in Theorem \ref{mainVk}. Then, we establish an isomorphism of graded $\mathbb{E}^{*,*}(K)$-modules:
  \[
    \mathbb{E}^{*,*}(V_k(\af^n)) = \bigoplus_{I \subset N_{n,k}} \mathbb{E}_I b_I
  \]
  where
  \[
    \mathbb{E}_I= \begin{cases}
      \mathbb{E}^{*,*}(K), & \text{if $I = I_F$}, \\
      (\mathbb{E}/\eta)^{*,*}(K), & \text{if $\mathrm{Max}(I_T)$ is even}, \\
      {^\eta\mathbb{E}}^{*,*}(K), & \text{if $\mathrm{Max}(I_T)$ is odd}. \\ 
    \end{cases}
  \]
\end{theorem}

We now proceed to prove the conjecture stated in \cite[Conj 1.4]{peng2023milnorwitt}. Recall that $ C_\eta(S) $ represent the cone of the morphism $ \Mt(S) \xr{\eta} \Mt(S)(1)[1] $, and

\begin{eqnarray*}
\wt{\mathbf{HS}}_{2k} & := & \Mt(\afnz{2k}) \cong \Mt(S) \oplus \Mt(S)(2k)[4k-1], \\
\wt{\mathbf{HS}}_{2k+1} & := & \Mt(S) \oplus C_\eta(S)(2k)[4k] \oplus \Mt(S)(4k+1)[8k]. 
\end{eqnarray*}

\begin{theorem}
	\label{mwDecompVk}
For any $ i, j \in \mathbb{N} $, we obtain an MW-motivic decomposition of the following forms:
	\[
	\Mt(V_{2j}(\af^{2i})) \cong \wt{\mathbf{HS}}_{2i} \otimes  \wt{\mathbf{HS}}_{2i-1} \otimes \wt{\mathbf{HS}}_{2i-3} \otimes \ldots \otimes \wt{\mathbf{HS}}_{2i-1-2(j-2)} \otimes \Mt(\afnz{2i+1-2j})
	\]
	\[
	\Mt(V_{2j+1}(\af^{2i})) \cong \wt{\mathbf{HS}}_{2i} \otimes \wt{\mathbf{HS}}_{2i-1} \otimes \wt{\mathbf{HS}}_{2i-3} \otimes \ldots \otimes \wt{\mathbf{HS}}_{2i-1-2(j-1)}
	\]
	\[
	\Mt(V_{2j}(\af^{2i+1})) \cong \wt{\mathbf{HS}}_{2i+1} \otimes \wt{\mathbf{HS}}_{2i-1} \otimes \ldots \otimes \wt{\mathbf{HS}}_{2i+1-2(j-1)}
	\]
	\[
	\Mt(V_{2j+1}(\af^{2i+1})) \cong \wt{\mathbf{HS}}_{2i+1} \otimes \wt{\mathbf{HS}}_{2i-1} \otimes \ldots \otimes \wt{\mathbf{HS}}_{2i+1-2(j-1)} \otimes \Mt(\afnz{2i+1-2j}).
	\]
\end{theorem}

\begin{proof}

We will only prove the case when $ S = \Spec K $. 

Let $ Y_{k,n} $ denote the right-hand side of the isomorphism in the theorem. By definition, we have the following cofiber sequences for $ Y_{k,n} $:

\[
  Y_{k-1,n-1} \to Y_{k,n} \to \tbZ(n)[2n-1] \otimes Y_{k-1,n-1} \xr{0}
\]
for $ n $ even, and

\[
  Y_{k-1,n-1} \to Y_{k,n} \to \tbZ(n)[2n-1] \otimes Y_{k-1,n-1} \xr{\beta}
\]
for $ n $ odd, where $ \beta $ is defined as follows:

\[
 \tbZ(n)[2n-1] \otimes ( Y_{k-2,n-2} \oplus  Y_{k-2,n-2}(n-1)[2n-3] ) 
 \xr{ \begin{bmatrix}
     0 & \eta \\ 
     0 & 0
\end{bmatrix} }
Y_{k-2,n-2}[1] \oplus  Y_{k-2,n-2}(n-1)[2n-2].
\]

From Theorem \ref{EmainVk}, we observe that for coconnected $ \mathbb{E} $, the following holds:

\[
  \rHom_{\DMt(K)}(\Mt(V_{k}(\af^{n})), \mathbb{E}(q)[p]) \cong \rHom_{\DMt(K)}(Y_{k,n}, \mathbb{E}(q)[p]).
\]

It is noteworthy that $ Y_{k,n} $ is a direct sum of $ \tbZ(q)[p] $ and $ C_{\eta}(q)[p] $. In particular, we have

\[
  \rHom_{\DMt(K)}(\Mt(V_{k}(\af^{n})), Y_{k,n}) \cong \rHom_{\DMt(K)}(Y_{k,n}, Y_{k,n}).
\]

Let $ \iota_{k,n} : \Mt(V_{k}(\af^{n})) \to Y_{k,n} $ denote the morphism corresponding to the identity of $ Y_{k,n} $. We now need to prove inductively that $ \iota_{k,n} $ is an isomorphism. For $ k=1 $, this is evident, and for $ k=2 $, the result follows from \cite[Cor 3.9]{peng2023milnorwitt}. 

For $ V_k(\af^n) $, we have the following cofiber sequence in $ \mathcal{H}(K) $ (see \cite[\S 3.2]{peng2023milnorwitt}):

\[
V_{k-1}(\af^{n-1})_+ \to V_k(\af^n)_+ \to \afnz{n} \wedge V_{k-1}(\af^{n-1})_+.
\]

When $ n $ is even, one can construct the splitting projection using that of $ Y_{k,n} $:

\[
  \rHom_{\DMt(K)}(\Mt(V_{k}(\af^{n})), \Mt(V_{k-1}(\af^{n-1}))) \cong \rHom_{\DMt(K)}(\Mt(V_{k}(\af^{n})), Y_{k-1,n-1}) \cong \rHom_{\DMt(K)}(Y_{k,n}, Y_{k-1,n-1}).
\]

Thus, it follows that

\[
  \Mt(V_{k}(\af^{n})) \cong \Mt(V_{k-1}(\af^{n-1})) \oplus \Mt(V_{k-1}(\af^{n-1}))(n)[2n-1] \cong Y_{k-1,n-1} \oplus Y_{k-1,n-1}(n)[2n-1] \cong Y_{k,n}.
\]

For $ n $ odd, by employing the same argument used in the proof of \cite[Prop 3.13]{peng2023milnorwitt}, we arrive at the following cofiber sequences in $ \DMt(K) $:

The equations presented elucidate the relationships between various mathematical structures in the study. Specifically, we have the following mappings:

\[
  \Mt( V_{k-2}(\af^{n-2}) ) \to \Mt( V_k(\af^n) ) \to Q_{k,n}  
\]
and
\[
  C_{\eta}\otimes \Mt( V_{k-2}(\af^{n-2}) )(n-1)[2n-2] \to Q_{k,n} \to \tbZ(n)[2n-1] \otimes \Mt( V_{k-2}(\af^{n-2}) )(n-1)[2n-3]
\]

In a similar vein, we can construct the splitting projections, leading us ultimately to the isomorphisms:

\begin{eqnarray*}
  & \Mt( V_k(\af^n) )  \cong  \Mt( V_{k-2}(\af^{n-2}) ) \oplus  \left( C_{\eta}\otimes \Mt( V_{k-2}(\af^{n-2}) )(n-1)[2n-2] \right) \oplus  \Mt( V_{k-2}(\af^{n-2}) )(2n-1)[4n-4] \\
  &                 \cong  Y_{k-2,n-2} \oplus  \left( C_{\eta}\otimes Y_{k-2,n-2}(n-1)[2n-2] \right) \oplus  Y_{k-2,n-2}(2n-1)[4n-4] \cong Y_{k,n}
\end{eqnarray*}

\end{proof}

\bibliographystyle{plain}
\bibliography{ref}
\end{document}